\documentclass[12pt]{article}
\usepackage[a4paper,margin=1in]{geometry}
\usepackage{amssymb}
\usepackage{amsmath}
\usepackage{amsfonts}
\usepackage{amsthm}
\usepackage{color}
\usepackage{float}
\usepackage[all]{xy}
\usepackage{scalerel}
\usepackage{mathabx}
\usepackage{delimset}

\newcommand{\C}{\mathbb{C}}

\newcommand{\N}{\mathbb{N}}
\newcommand{\Q}{\mathbb{Q}}
\newcommand{\R}{\mathbb{R}}

\newcommand{\fibonomial}{\genfrac{\{}{\}}{0pt}{}}

\newtheorem{theorem}{Theorem}

\newtheorem{definition}{Definition}
\newtheorem{proposition}{Proposition}

\newtheorem{example}{Example}

\newtheorem{remark}{Remark}

\begin{document}

\title{\textbf{Deformed Newton's $(s,t)$-Binomial Series and Generating Functions of Generalized Central Binomial Coefficients and Generalized Catalan Numbers}}
\author{Ronald Orozco L\'opez}
\newcommand{\Addresses}{{
  \bigskip
  \footnotesize

  \textit{E-mail address}, R.~Orozco: \texttt{rj.orozco@uniandes.edu.co}
  
}}

\maketitle
\tableofcontents

\begin{abstract}
We define the deformed $(s,t)$-binomial formula and the deformed Newton $(s,t)$-binomial series, and we will use it to establish the generating functions of the generalized central binomial coefficients and the generalized Catalan numbers. 
\end{abstract}


\section{Introduction}

The Catalan numbers are defined by
\begin{equation*}
    C_{n}=\frac{1}{n+1}\binom{2n}{n}
\end{equation*}
for $n\geq0$. It is a known fact that the generating function of Catalan's numbers is
\begin{equation}\label{eqn_gfc}
    C(x)=\frac{1-\sqrt{1-4x}}{2x}.
\end{equation}
The generalized Catalan numbers \cite{ekhad, Sa1, bennet} are defined by
\begin{equation}
    C_{\brk[c]{n}}=\frac{1}{\brk[c]{n+1}_{s,t}}\fibonomial{2n}{n}_{s,t}.
\end{equation}
The generalized Catalan numbers are polynomials in $s,t$ with nonnegative integral coefficients \cite{ekhad}, and its combinatorial interpretation is given in \cite{bennet}. However, a generating function, analogous to Eq.(\ref{eqn_gfc}), for the polynomials $C_{\{n\}}$ is not known. This is the main goal of this paper. Therefore all the results are original and useful.

In this paper, we introduce the deformed $(s,t)$-analog of the binomial theorem
\begin{equation*}
    (x+y)^{n}=\sum_{k=0}^{n}\binom{n}{k}x^{n-k}y^{k}
\end{equation*}
as the equation given by
\begin{equation}\label{eqn_st_bin_theo}
    (x\oplus_{u,v}y)^{(n)}=\sum_{k=0}^{n}\fibonomial{n}{k}_{s,t}u^{\binom{n-k}{2}}v^{\binom{k}{2}}x^{n-k}y^{k}.
\end{equation}
When make $s=u=q$, $t=v=1$, $a=t$ and $x=1$, $y=t$ in Eq.(\ref{eqn_st_bin_theo}), we obtain the $q$-theorem binomial
\begin{equation*}
    \prod_{k=0}^{n-1}(1+q^kt)=\sum_{k=0}^{n}q^{k(k-1)/2}\binom{n}{k}_{q}t^k
\end{equation*}
where $\binom{n}{k}_{q}$ are the Gaussian binomial coefficients \cite{gauss_1,gauss_2}.

On the other hand, if in Eq.(\ref{eqn_st_bin_theo}) we extend the defining range of $n$ to complex numbers, then we obtain the $(u,v)$-deformed Newton's $(s,t)$-binomial series
\begin{equation}\label{eqn_def_newton}
    (x\oplus_{u,v}y)^{(\alpha)}=\sum_{k=0}^{\infty}\fibonomial{\alpha}{k}_{s,t}u^{\binom{\alpha-k}{2}}v^{\binom{k}{2}}x^{\alpha-k}y^{k}.
\end{equation}

This article is divided as follows. Section \ref{sec_pre} addresses basic and necessary results for developing this article. Section 3 introduces the deformed Newton $(s,t)$-binomial series and establishes its algebraic properties. In Section 4, we express the fibonomial coefficients $\fibonomial{1/2}{n}$ and $\fibonomial{-1/2}{n}$ in term of central $(s,t)$-binomial coefficients. Finally, we conclude our paper by focusing on finding the generating functions of the generalized binomial coefficients introduced above.

\section{Preliminaries}\label{sec_pre}

\subsection{Generalized Fibonacci polynomials}

The generalized Fibonacci numbers on the parameters $s,t$ are defined by
\begin{equation}\label{eqn_def_fibo}
    \brk[c]{n+2}_{s,t}=s\{n+1\}_{s,t}+t\{n\}_{s,t}
\end{equation}
with initial values $\brk[c]{0}_{s,t}=0$ and $\brk[c]{1}_{s,t}=1$, where $s\neq0$ and $t\neq0$. In \cite{Sa1} this sequence is called the generalized Lucas sequence. In this paper, the generalized Lucas numbers $\brk[a]{n}_{s,t}$ on the parameters $s,t$ are defined recursively by
\begin{equation*}
    \brk[a]{n+2}_{s,t}=s\brk[a]{n+1}_{s,t}+t\brk[a]{n}_{s,t}
\end{equation*}
for $n\geq2$, with the initial conditions $\brk[a]{0}_{s,t}=2$ and $\brk[a]{1}_{s,t}=s$. The Lucas sequence $L_{n}=(2,1,3,4,7,11,18,29,\ldots)$ is obtained when $s=1$, $t=1$ and is not a special case of the sequence in Eq.(\ref{eqn_def_fibo}). Therefore, we will insist on calling to $\brk[c]{n}_{s,t}$ the generalized Fibonacci numbers and we will reserve the name of generalized Lucas numbers for the sequence $\brk[a]{n}_{s,t}$.

Below are some important specializations of generalized Fibonacci and Lucas numbers.
\begin{enumerate}
    \item When $s=2,t=-1$, then $\brk[c]{n}_{2,-1}=n$, the positive integer and $\brk[a]{n}_{2,-1}=2$, for all $n\in\N$.
    \item When $s=1,t=1$, then $\brk[c]{n}_{1,1}=F_n$, the Fibonacci numbers and $\brk[a]{2}_{1,1}=L_{n}$, the Lucas numbers.
    \item When $s=2,t=1$, then $\brk[c]{n}_{2,1}=P_n$, where $P_n=(0,1,2,5,12,29,70,169,408,\ldots)$ are the Pell numbers and $\brk[a]{n}_{2,1}=Q_{n}$, where $Q_n=(2,2,6,14,34,82,198,\ldots)$ are the Pell-Lucas numbers.
    \item When $s=1,t=2$, then $\brk[c]{n}_{1,2}=J_n$, where $J_n=(0,1,1,2,3,5,11,21,43,85,\ldots)$ are the Jacosbthal numbers
    and $\brk[a]{n}_{1,2}=j_{n}$, where $j_{n}=(2,1,5,7,17,31,65,127,\ldots)$ are the Jacobthals-Lucas numbers
    \item When $s=p+q,t=-pq$, then $\brk[c]{n}_{p+q,-pq}=\brk[s]{n}_{p,q}=\frac{p^n-q^n}{p-q}$ are the $(p,q)$-numbers
    and $\brk[a]{n}_{p+q,-pq}=[2n]_{p,q}/[n]_{p,q}$.
    \item When $s=2t,t=-1$, then $\brk[c]{n}_{2t,-1}=U_{n-1}(t)$, where $U_n(t)$ are the Chebysheff polynomials of the second kind, with $U_{-1}(t)=0$ and $\brk[a]{n}_{2t,-1}=2T_{n}(t)$, where $T_{n}(t)$ are the Chebysheff polynomials of the first kind.
    \item When $s=3,t=-2$, then $\brk[c]{n}_{3,-2}=M_n$, where $M_n=2^n-1$ are the Mersenne numbers and $\brk[a]{n}_{3,-2}=2^{n}+1$.
    \item When $s=P,t=-Q$, then $\brk[c]{n}_{P,-Q}=U_n(P,-Q)$, where $U_n(P,-Q)$ is the $(P,-Q)$-Fibonacci sequence, with $P,Q$ integer numbers, and $\brk[a]{n}_{P,-Q}=V(n)$ is the $(P,-Q)$-Lucas sequence. If $Q=-1$, then the sequence $U_n(P,-1)$ reduces to the $P$-Fibonacci sequence. If $s=x$ and $t=1$, we obtain the Fibonacci polynomials
\end{enumerate}

The $(s,t)$-Fibonacci constant is the ratio toward which adjacent $(s,t)$-Fibonacci polynomials tend. This is the only positive root of $x^{2}-sx-t=0$. We will let $\varphi_{s,t}$ denote this constant, where
\begin{equation*}
    \varphi_{s,t}=\frac{s+\sqrt{s^{2}+4t}}{2}
\end{equation*}
and
\begin{equation*}
    \varphi_{s,t}^{\prime}=s-\varphi_{s,t}=-\frac{t}{\varphi_{s,t}}=\frac{s-\sqrt{s^{2}+4t}}{2}.
\end{equation*}
Some Generalized Fibonacci and Lucas polynomials are given in the table following.
\begin{table}[H]
    \centering
    \begin{tabular}{c|c|c}
         $n$& $\brk[c]{n}_{s,t}$&$\brk[a]{n}_{s,t}$\\
         \hline
         $0$& $0$ & $2$ \\
         $1$& $1$ & $s$ \\
         $2$& $s$ & $s^2+2t$ \\
         $3$& $s^2+t$ & $s^3+3st$\\
         $4$& $s(s^2+2t)$ & $s^4+4s^2t+2t^2$\\
         $5$& $s^4+3s^2t+t^2$ & $s^5+5s^3t+5st^2$\\
         $6$& $s^5+4s^3t+3st^2$ & $s^6+6s^4t+11s^2t^2+2t^3$
    \end{tabular}
    \label{tab:my_label}
\end{table}

It is a known fact \cite{Sa1} that
\begin{equation}\label{eqn_2n}
    \brk[c]{2n}_{s,t}=\brk[a]{n}_{s,t}\brk[c]{n}_{s,t}.
\end{equation}

\begin{definition}\label{def_gff}
For all $\alpha\in\C$ we define the generalized Fibonacci functions $\brk[c]{\alpha}_{s,t}$ as
\begin{equation}
    \brk[c]{\alpha}_{s,t}=\frac{\varphi_{s,t}^{\alpha}-\varphi_{s,t}^{\prime\alpha}}{\varphi_{s,t}-\varphi_{s,t}^\prime}
\end{equation}
and the generalized Lucas functions $\brk[a]{\alpha}_{s,t}$ as
\begin{equation}
    \brk[a]{\alpha}_{s,t}=\varphi_{s,t}^{\alpha}+\varphi_{s,t}^{\prime\alpha}.
\end{equation}
\end{definition}

\begin{example}
For all $s,t\in\R$,
\begin{equation*}
    \Big\{\frac{1}{2}\Big\}_{s,t}=\frac{1}{\sqrt{\frac{s+\sqrt{s^2+4t}}{2}}+\sqrt{\frac{s-\sqrt{s^2+4t}}{2}}}
\end{equation*}
and
\begin{equation*}
    \Big\langle\frac{1}{2}\Big\rangle_{s,t}=\sqrt{\frac{s+\sqrt{s^2+4t}}{2}}+\sqrt{\frac{s-\sqrt{s^2+4t}}{2}}.
\end{equation*}
\end{example}
The negative $(s,t)$-Fibonacci functions are given by
\begin{equation}\label{eqn_neg_fibo}
\brk[c]{-\alpha}_{s,t}=-(-t)^{-\alpha}\brk[c]{\alpha}_{s,t}
\end{equation}
for all $\alpha\in\R$.

\subsection{Generalized $(s,t)$-fibonomial coefficients}
The $(s,t)$-fibonomial coefficients are define by
\begin{equation*}
    \fibonomial{n}{k}_{s,t}=\frac{\brk[c]{n}_{s,t}!}{\brk[c]{k}_{s,t}!\brk[c]{n-k}_{s,t}!},
\end{equation*}
where $\brk[c]{n}_{s,t}!=\brk[c]{1}_{s,t}\brk[c]{2}_{s,t}\cdots\brk[c]{n}_{s,t}$ is the $(s,t)$-factorial or generalized fibotorial.

The $(s,t)$-fibonomial coefficients satisfy the following Pascal recurrence relationships. For $1\leq k\leq n-1$ it holds that
\begin{align*}
    \fibonomial{n+1}{k}_{s,t}&=\varphi_{s,t}^{k}\fibonomial{n}{k}_{s,t}+\varphi_{s,t}^{\prime(n+1-k)}\fibonomial{n}{k-1}_{s,t},\\
    &=\varphi_{s,t}^{\prime(k)}\fibonomial{n}{k}_{s,t}+\varphi_{s,t}^{n+1-k}\fibonomial{n}{k-1}_{s,t}.
\end{align*}

\begin{definition}\label{def_fibo_real}
For all $\alpha\in\C$ the generalized $(s,t)$-fibonomial coefficient is
\begin{equation}
    \fibonomial{\alpha}{k}_{s,t}=\frac{\brk[c]{\alpha}_{s,t}\brk[c]{\alpha-1}_{s,t}\cdots\brk[c]{\alpha-k+1}_{s,t}}{\brk[c]{k}_{s,t}!}.
\end{equation}
\end{definition}
Equally, the generalized $(s,t)$-fibonomial coefficient satisfy the following relationships:
\begin{align}
    \fibonomial{\alpha+1}{k}_{s,t}&=\varphi_{s,t}^{k}\fibonomial{\alpha}{k}_{s,t}+\varphi_{s,t}^{\prime(\alpha+1-k)}\fibonomial{\alpha}{k-1}_{s,t}\label{prop_pascal1},\\
    &=\varphi_{s,t}^{\prime(k)}\fibonomial{\alpha}{k}_{s,t}+\varphi_{s,t}^{\alpha+1-k}\fibonomial{\alpha}{k-1}_{s,t}\label{prop_pascal2}.
\end{align}

\begin{proposition}\label{prop_fibo_def}
For all nonzero $a\in\C$,
    \begin{equation}
        \fibonomial{\alpha}{n}_{as,a^2t}=a^{n(\alpha-n)}\fibonomial{\alpha}{n}_{s,t}.
    \end{equation}
\end{proposition}
\begin{proof}
As $\brk[c]{n}_{as,a^2t}=a^{n-1}\brk[c]{n}_{s,t}$, then $\brk[c]{n}_{as,a^2t}!=a^{\binom{b}{2}}\brk[c]{n}_{s,t}!$. Therefore,
\begin{align*}
    \fibonomial{\alpha}{n}_{as,a^2t}&=\frac{\brk[c]{\alpha}_{as,a^at}\brk[c]{\alpha-1}_{as,a^at}\cdots\brk[c]{\alpha-n+1}_{as,a^at}}{\brk[c]{n}_{as,a^at}!}\\
    &=\frac{a^{\alpha-1+\alpha-2+\cdots+\alpha-n}\brk[c]{\alpha}_{s,t}\brk[c]{\alpha-1}_{s,t}\cdots\brk[c]{\alpha-n+1}_{s,t}}{a^{\binom{n}{2}}\brk[c]{n}_{s,t}!}\\
    &=\frac{a^{n\alpha-\binom{n+1}{2}}\brk[c]{\alpha}_{s,t}\brk[c]{\alpha-1}_{s,t}\cdots\brk[c]{\alpha-n+1}_{s,t}}{a^{\binom{n}{2}}\brk[c]{n}_{s,t}!}\\
    &=a^{n(\alpha-n)}\frac{\brk[c]{\alpha}_{s,t}\brk[c]{\alpha-1}_{s,t}\cdots\brk[c]{\alpha-n+1}_{s,t}}{\brk[c]{n}_{s,t}!}.
\end{align*}
The proof is completed.
\end{proof}

\begin{proposition}\label{prop_neg_bino}
The negative $(s,t)$-fibonomial polynomial is
\begin{equation*}
    \fibonomial{-\alpha}{k}_{s,t}=(-1)^k(-t)^{-\alpha k-\binom{k}{2}}\fibonomial{\alpha+k-1}{k}_{s,t}.
\end{equation*}
\end{proposition}
\begin{proof}
By using Eq.(\ref{eqn_neg_fibo})
\begin{align*}
    \fibonomial{-\alpha}{k}_{s,t}&=\frac{\brk[c]{-\alpha}_{s,t}\brk[c]{-\alpha-1}_{s,t}\cdots\brk[c]{-\alpha-k+1}_{s,t}}{\brk[c]{k}_{s,t}!}\\
    &=\frac{[-(-t)^{-\alpha}\brk[c]{\alpha}_{s,t}][-(-t)^{-(\alpha+1)}\brk[c]{\alpha+1}_{s,t}]\cdots[-(-t)^{-(\alpha+k-1)}\brk[c]{\alpha+k-1}_{s,t}]}{\brk[c]{k}_{s,t}!}\\
    &=\frac{(-1)^{k}(-t)^{-k\alpha-\binom{k}{2}}\brk[c]{\alpha+k-1}_{s,t}\brk[c]{\alpha+k-2}_{s,t}\cdots\brk[c]{\alpha+1}_{s,t}\brk[c]{\alpha}_{s,t}}{\brk[c]{k}_{s,t}!}\\
    &=(-1)^{k}(-t)^{-k\alpha-\binom{k}{2}}\fibonomial{\alpha+k-1}{k}_{s,t}.
\end{align*}
The proof is completed.
\end{proof}

\subsection{Derivative on generalized Fibonacci polynomials}

\begin{definition}
Set $s,t\in\R$, $s\neq0$, $t\neq0$. If $s^2+4t\neq0$, define the $(s,t)$-derivative $\mathbf{D}_{s,t}$ of the function $f(x)$ as
\begin{equation}
(\mathbf{D}_{s,t}f)(x)=
\begin{cases}
\frac{f(\varphi_{s,t}x)-f(\varphi_{s,t}^{\prime}x)}{(\varphi_{s,t}-\varphi_{s,t}^{\prime})x},&\text{ if }x\neq0;\\
f^{\prime}(0),&\text{ if }x=0
\end{cases}
\end{equation}
provided $f(x)$ differentiable at $x=0$. If $s^2+4t=0$, $t<0$, define the $(\pm i\sqrt{t},t)$-derivative of the function $f(x)$ as
\begin{equation}
    (\mathbf{D}_{\pm i\sqrt{t},t}f)(x)=f^{\prime}(\pm i\sqrt{t}x).
\end{equation}
If $(\mathbf{D}_{s,t}f)(x)$ exist at $x=a$, then $f(x)$ is $(s,t)$-differentiable at $a$.
\end{definition}

\begin{example}
For all $s,t\in\R$, $s\neq0$, $t\neq0$
\begin{equation*}
    \mathbf{D}_{s,t}x^n=\frac{(\varphi_{s,t}x)^n-(\varphi_{s,t}^\prime x)^n}{(\varphi_{s,t}-\varphi_{s,t}^\prime)x}=\brk[c]{n}_{s,t}x^{n-1},\ n\in\N
\end{equation*}
and $x^n$ is $(s,t)$-differentiable in $\R$.
\end{example}

\begin{example}
For all $s,t\in\R$, $s\neq0$, $t\neq0$ and $x\neq0$ the $(s,t)$-derivative of $f(x)=e^x$ is
\begin{equation*}
    \mathbf{D}_{s,t}e^x=\frac{e^{\varphi_{s,t}x}-e^{\varphi_{s,t}^{\prime}x}}{(\varphi_{s,t}-\varphi_{s,t}^\prime)x}=
    \begin{cases}
    2\frac{e^{sx/2}}{\sqrt{s^2+4t}x}\sinh\left(\frac{\sqrt{s^2+4t}}{2}x\right),&\text{ if }s^2+4t>0;\\
    e^{\pm i\sqrt{t}x},&\text{ if }s^2+4t=0,t<0;\\
    2\frac{e^{sx/2}}{\sqrt{-s^2-4t}x}\sin\left(\frac{\sqrt{-s^2-4t}}{2}x\right),&\text{ if }s^2+4t<0.
    \end{cases}
\end{equation*}
As $f^{\prime}(0)=1$, then $\mathbf{D}_{s,t}e^x$ is continuous in $\R$ and $(s,t)$-differentiable in $\R$.
\end{example}

\begin{proposition}
For all $\alpha,\beta\in\R$,
    \begin{enumerate}
        \item $\mathbf{D}_{s,t}(\alpha f+\beta g)=\alpha\mathbf{D}_{s,t}f+\beta\mathbf{D}_{s,t}g$.
        \item $\mathbf{D}_{s,t}(f(x)g(x))=f(\varphi_{s,t}x)\mathbf{D}_{s,t}g(x)+g(\varphi_{s,t}^\prime x)\mathbf{D}_{s,t}f(x)$.
        \item $\mathbf{D}_{s,t}(f(x)g(x))=f(\varphi_{s,t}^{\prime}x)\mathbf{D}_{s,t}g(x)+g(\varphi_{s,t}x)\mathbf{D}_{s,t}f(x)$.
        \item
        \begin{align*}
            \mathbf{D}_{s,t}\left(\frac{f(x)}{g(x)}\right)&=\frac{g(\varphi_{s,t}x)\mathbf{D}_{s,t}f(x)-f(\varphi_{s,t}x)\mathbf{D}_{s,t}g(x)}{g(\varphi_{s,t}x)g(\varphi_{s,t}^{\prime}x)}.
        \end{align*}
        \item
        \begin{align*}
            \mathbf{D}_{s,t}\left(\frac{f(x)}{g(x)}\right)&=\frac{g(\varphi_{s,t}^{\prime}x)\mathbf{D}_{s,t}f(x)-f(\varphi_{s,t}^{\prime}x)\mathbf{D}_{s,t}g(x)}{g(\varphi_{s,t}x)g(\varphi_{s,t}^{\prime}x)}.
        \end{align*}
    \end{enumerate}
\end{proposition}

\section{Deformed Newton $(s,t)$-binomial series}

\subsection{Definition}

\begin{definition}
For all $\alpha\in\C$ and $s\neq0$, $t\neq0$ define the $(u,v)$-deformed $(s,t)$-binomial series, for $x,y$ commuting, as
\begin{equation}\label{eqn_nbs}
    (x\oplus_{u,v}y)_{s,t}^{(\alpha)}=\sum_{n=0}^{\infty}\fibonomial{\alpha}{n}_{s,t}u^{\binom{\alpha-n}{2}}v^{\binom{n}{2}}x^{\alpha-n}y^{n}.
\end{equation}    
We define the $(u,v)$-deformed $(s,t)$-analogue of $(x-y)^\alpha$ by
\begin{equation*}
    (x\ominus_{u,v}y)_{s,t}^{(\alpha)}\equiv (x\oplus_{u,v}(-y))_{s,t}^{(\alpha)}.
\end{equation*}
\end{definition}

\begin{definition}
For all nonzero complex numbers $u,v$, $u,v\in\C$, define the $(u,v)$-deformed $(s,t)$-power, for $x,y$ commuting, as
\begin{equation*}
(x\oplus_{u,v}y)_{s,t}^{(n)}=\sum_{k=0}^{n}\fibonomial{n}{k}_{s,t}u^{\binom{n-k}{2}}v^{\binom{k}{2}}x^{n-k}y^{k}\ \ \ \ n\geq0.
\end{equation*}
\end{definition}

\begin{example}
\begin{align*}
    (x\oplus_{u,v}y)_{s,t}^{(0)}&=1,\\
    (x\oplus_{u,v}y)_{s,t}^{(1)}&=x+y,\\
    (x\oplus_{u,v}y)_{s,t}^{(2)}&=ux^{2}+\brk[c]{2}_{s,t}xy+vy^{2},\\
    (x\oplus_{u,v}y)_{s,t}^{(3)}&=u^{3}x^{3}+\brk[c]{3}_{s,t}ux^{2}y+\brk[c]{3}_{s,t}vxy^{2}+v^{3}y^{3},\\
    (x\oplus_{u,v}y)_{s,t}^{(4)}&=u^{6}x^{4}+\brk[c]{4}_{s,t}u^{3}x^{3}y+\brk[c]{3}_{s,t}\brk[a]{2}_{s,t}uvx^{2}y^{2}+\brk[c]{4}_{s,t}v^{3}xy^{3}+v^{6}y^{4}.
\end{align*}    
\end{example}

In post quantum calculus, or $(p,q)$-calculus, the $(p,q)$-powers are defined by Sadjang in \cite{nji} by
\begin{align}
(x\ominus a)_{p,q}^{n}&=
\begin{cases}
1,& \text{ if }n=0;\\
\prod_{k=0}^{n-1}(p^{k}x-q^{k}a),& \text{ if }n\geq1.
\end{cases}\label{eqn_pq_binomial1}\\
(x\oplus a)_{p,q}^{n}&=
\begin{cases}
1,& \text{ if }n=0;\\
\prod_{k=0}^{n-1}(p^{k}x+q^{k}a),& \text{ if }n\geq1.
\end{cases}\label{eqn_pq_binomial2}
\end{align}
In $q$-calculus the $q$-shifted factorial is
\begin{equation*}
    (a;q)_n=
    \begin{cases}
        1,& \text{ if }n=0;\\
        \prod_{k=0}^{n-1}(1-aq^{k}),& \text{ if }n\geq1.
    \end{cases}
\end{equation*}
When $u=\varphi$ and $v=\varphi^\prime$, then 
\begin{equation*}
    (x\ominus_{\varphi,\varphi^\prime}a)_{s,t}^{(n)}=(x\ominus a)_{\varphi,\varphi^{\prime}}^{n}=\varphi_{s,t}^{\binom{n}{2}}x^n(a/x;q)_{n}
\end{equation*}
where $x\neq0$ and $q=\varphi_{s,t}^{\prime}/\varphi_{s,t}$.
Therefore, the $(u,v)$-deformed $(s,t)$-power generalize the $(p,q)$-power. 
Also, the $(u,v)$-deformed $(s,t)$-power generalize the $(p,q)$-Gauss Binomial formula 
\begin{equation}\label{eqn_gauss_bin}
    (x\ominus a)_{p,q}^{n}=\sum_{k=0}^{n}\binom{n}{k}_{p,q}p^{\binom{n-k}{2}}q^{\binom{k}{2}}x^{n-k}(-a)^{k},
\end{equation}
the Jackson-Hanh-Cigler $q$-addition \cite{jackson_8, cigler, hahn}
\begin{equation*}
    (x\boxplus_qy)^n=\sum_{k=0}^{n}\binom{n}{k}_qq^{\binom{k}{2}}x^{k}y^{n-k},
\end{equation*}
the Ward-Al-Salam $q$-addition \cite{alsalam,ward}
\begin{equation*}
    (x\oplus_qy)^n=\sum_{k=0}^{n}\binom{n}{k}_qx^{k}y^{n-k},
\end{equation*}
and the Golden Binomial formula \cite{pashaev_1}
\begin{equation*}
    (x+y)_{F}^n=\sum_{k=0}^{n}\binom{n}{k}_{F}(-1)^{\binom{k}{2}}x^{k}y^{n-k}.
\end{equation*}
For $s=2$ and $t=-1$, the $(u,v)$-deformed $n$-th power of $x+y$ is
\begin{equation}
    (x\oplus_{u,v}y)^{(n)}=\sum_{k=0}^{k}\binom{n}{k}u^{\binom{n-k}{2}}v^{\binom{n}{2}}x^{n-k}y^{k}.
\end{equation}
Eqs.(\ref{eqn_pq_binomial1}),(\ref{eqn_pq_binomial2}) are extended to
\begin{align*}
(x\ominus_{\varphi,\varphi^\prime}y)_{s,t}^{(\infty)}&=\prod_{k=0}^{\infty}(\varphi_{s,t}^{k}x-\varphi_{s,t}^{\prime k}y),\\
(x\oplus_{\varphi,\varphi^\prime}y)_{s,t}^{(\infty)}&=\prod_{k=0}^{\infty}(\varphi_{s,t}^{k}x+\varphi_{s,t}^{\prime k}y).
\end{align*}
It is easy to verify that
\begin{equation}\label{eqn_power_inf}
    (x\oplus_{\varphi,\varphi^\prime}y)_{s,t}^{(n)}=\frac{(x\oplus_{\varphi,\varphi^\prime}y)_{s,t}^{(\infty)}}{(x\varphi_{s,t}^n\oplus_{\varphi,\varphi^\prime}y\varphi_{s,t}^{\prime n})_{s,t}^{(\infty)}}
\end{equation}
for $n\in\N$.
Then Eq.(\ref{eqn_power_inf}) can be extended for all complex number as
\begin{equation}
    (x\oplus_{\varphi,\varphi^\prime}y)_{s,t}^{(\alpha)}=\frac{(x\oplus_{\varphi,\varphi^\prime}y)_{s,t}^{(\infty)}}{(x\varphi_{s,t}^\alpha\oplus_{\varphi,\varphi^\prime}y\varphi_{s,t}^{\prime \alpha})_{s,t}^{(\infty)}},
\end{equation}
$\alpha\in\C$.

\subsection{Properties}

The following are properties satisfied by the formula $(u,v)$-deformed $(s,t)$-binomial.
\begin{theorem}\label{theo_propi_uvFbinom}
For all $\alpha,x,y,z,u,v\in\C$ the $(u,v)$-deformed $(s,t)$-binomial series has the following properties
\begin{enumerate}
    \item $(x\oplus_{u,v}y)_{s,t}^{(\alpha+1)}=x(ux\oplus_{u,v}\varphi_{s,t}y)_{s,t}^{(\alpha)}+y(\varphi_{s,t}^{\prime}x\oplus_{u,v}vy)_{s,t}^{(\alpha)}$,
    \item $(x\oplus_{u,v}y)_{s,t}^{(\alpha+1)}=x(ux\oplus_{u,v}\varphi_{s,t}^{\prime}y)_{s,t}^{(\alpha)}+y(\varphi_{s,t}x\oplus_{u,v}vy)_{s,t}^{(\alpha)}$,
    \item $z^\alpha(x\oplus_{u,v}y)_{s,t}^{(\alpha)}=(zx\oplus_{u,v}zy)_{s,t}^{(\alpha)}$,
    \item $(x\oplus_{au,av}y)_{as,a^2t}^{(\alpha)}=a^{\binom{\alpha}{2}}(x\oplus_{u,v}y)_{s,t}^{(\alpha)}$,
    \item $(x\oplus_{u,v}y)_{s,t}^{(\alpha)}=(y\oplus_{v,u}x)_{s,t}^{(\alpha)}$,
    \item $(x\oplus_{u,v}0)^{(\alpha)}=u^{\binom{\alpha}{2}}x^\alpha$, for all $x\in\C$.
    \item $(0\oplus_{u,v}y)^{(\alpha)}=v^{\binom{\alpha}{2}}y^\alpha$, for all $y\in\C$.
\end{enumerate}
\end{theorem}
\begin{proof}
By Eq.(\ref{eqn_nbs})
\begin{align*}
    (x\oplus_{u,v}y)_{s,t}^{(\alpha+1)}&=\sum_{k=0}^{\infty}\fibonomial{\alpha+1}{k}_{s,t}u^{\binom{\alpha+1-k}{2}}v^{\binom{k}{2}}x^{\alpha+1-k}y^{k}.
\end{align*}
By extracting the first summand and by using Eq.(\ref{prop_pascal1}), we obtain
\begin{align*}
    &(x\oplus_{u,v}y)_{s,t}^{(\alpha+1)}\\
    &\hspace{1cm}=u^{\binom{\alpha+1}{2}}x^{\alpha+1}
    +\sum_{k=1}^{\infty}\left(\varphi_{s,t}^{k}\fibonomial{\alpha}{k}_{s,t}+\varphi_{s,t}^{\prime(\alpha-k+1)}\fibonomial{\alpha}{k-1}_{s,t}\right)v^{\binom{k}{2}}u^{\binom{\alpha+1-k}{2}}y^{k}x^{\alpha+1-k}\\
    &\hspace{1cm}=u^{\binom{\alpha+1}{2}}x^{\alpha+1}+\sum_{k=1}^{\infty}\varphi_{s,t}^{k}\fibonomial{\alpha}{k}_{s,t}v^{\binom{k}{2}}u^{\binom{\alpha+1-k}{2}}y^{k}x^{\alpha+1-k}\\
    &\hspace{2cm}+\sum_{k=1}^{\infty}\varphi_{s,t}^{\prime(\alpha-k+1)}\fibonomial{\alpha}{k-1}_{s,t}v^{\binom{k}{2}}u^{\binom{\alpha+1-k}{2}}y^{k}x^{\alpha+1-k}\\
    &\hspace{1cm}=u^{\binom{\alpha+1}{2}}x^{\alpha+1}
    +x\sum_{k=1}^{\infty}\fibonomial{\alpha}{k}_{s,t}v^{\binom{k}{2}}u^{\binom{\alpha-k}{2}}(\varphi_{s,t}y)^{k}(u x)^{\alpha-k}\\
    &\hspace{2cm}+\sum_{k=1}^{\infty}\varphi_{s,t}^{\prime(\alpha-k+1)}\fibonomial{\alpha}{k-1}_{s,t}v^{\binom{k}{2}}u^{\binom{\alpha+1-k}{2}}y^{k}x^{\alpha+1-k}.
\end{align*}
In addition, we rearrange the second summation for $k$
\begin{align*}
    (x\oplus_{u,v}y)_{s,t}^{(\alpha+1)}
    &=u^{\binom{\alpha+1}{2}}x^{\alpha+1}
    +x\sum_{k=1}^{\infty}\fibonomial{\alpha}{k}_{s,t}v^{\binom{k}{2}}u^{\binom{\alpha-k}{2}}(\varphi_{s,t}y)^{k}(u x)^{\alpha-k}\\
    &\hspace{1cm}+\sum_{k=0}^{\infty}\varphi_{s,t}^{\prime(\alpha-k)}\fibonomial{\alpha}{k}_{s,t}v^{\binom{k+1}{2}}u^{\binom{\alpha-k}{2}}y^{k+1}x^{\alpha-k}\\
    &=u^{\binom{\alpha+1}{2}}x^{\alpha+1}
    +x\sum_{k=1}^{\infty}\fibonomial{\alpha}{k}_{s,t}v^{\binom{k}{2}}u^{\binom{\alpha-k}{2}}(\varphi_{s,t}y)^{k}(u x)^{\alpha-k}\\
    &\hspace{1cm}+y\sum_{k=0}^{\infty}\fibonomial{\alpha}{k}_{s,t}v^{\binom{k}{2}}u^{\binom{\alpha-k}{2}}(v y)^{k}(\varphi_{s,t}^{\prime}x)^{\alpha-k}.
\end{align*}
Then
\begin{align*}
    &(x\oplus_{u,v}y)_{s,t}^{(\alpha+1)}\\
    &\hspace{1cm}=x\sum_{k=0}^{\infty}\fibonomial{\alpha}{k}_{s,t}v^{\binom{k}{2}}u^{\binom{\alpha-k}{2}}(\varphi_{s,t}y)^{k}(u x)^{\alpha-k}
    +y\sum_{k=0}^{\alpha}\fibonomial{\alpha}{k}_{s,t}v^{\binom{k}{2}}u^{\binom{\alpha-k}{2}}(vy)^{k}(\varphi_{s,t}^{\prime}x)^{\alpha-k}\\
    &\hspace{1cm}=x(ux\oplus_{u,v}\varphi_{s,t}y)_{s,t}^{(\alpha)}+y(\varphi_{s,t}^{\prime}x\oplus_{u,v}vy)_{s,t}^{(\alpha)}
\end{align*}
and thus we obtain the first statement. By using $(s,t)$-Pascal formula, Eq.(\ref{prop_pascal2}), we obtain 2. The proof of $3$ is trivial. From Proposition \ref{prop_fibo_def}
\begin{align*}
    (x\oplus_{au,av}y)_{as,a^2t}^{(\alpha)}&=\sum_{n=0}^{\infty}\fibonomial{\alpha}{n}_{as,a^2t}(au)^{\binom{\alpha-n}{2}}(av)^{\binom{n}{2}}x^{\alpha-n}y^{n}\\
    &=\sum_{n=0}^{\infty}a^{n(\alpha-1)+\binom{\alpha-n}{2}+\binom{n}{2}}\fibonomial{\alpha}{n}_{s,t}u^{\binom{\alpha-n}{2}}v^{\binom{n}{2}}x^{\alpha-n}y^{n}\\
    &=a^{\binom{\alpha}{2}}\sum_{n=0}^{\infty}\fibonomial{\alpha}{n}u^{\binom{\alpha-n}{2}}v^{\binom{n}{2}}x^{\alpha-n}y^{n}=a^{\binom{\alpha}{2}}(x\oplus_{u,v}y)_{s,t}^{(\alpha)}.
\end{align*}
Which leads us to the truth of 4. By changing $u$ to $v$ and $x$ to $y$ in the Eq.(\ref{eqn_nbs}) we find that
\begin{align*}
(x\oplus_{u,v}y)_{s,t}^{(\alpha)}&=\sum_{k=0}^{\infty}\fibonomial{\alpha}{k}_{s,t}v^{\binom{k}{2}}u^{\binom{\alpha-k}{2}}y^{k}x^{\alpha-k}\\
&=\sum_{k=0}^{\infty}\fibonomial{\alpha}{k}_{s,t}v^{\binom{\alpha-k}{2}}u^{\binom{k}{2}}y^{\alpha-k}y^{k}\\
&=(y\oplus_{v,u}x)_{s,t}^{(\alpha)}.
\end{align*}
Then it proves $5$. The statement 6 and 7 are trivial.
\end{proof}
We will now give a result on the convergence of Eq.(\ref{eqn_nbs}). Set $x\neq0$. As $(x\oplus_{u,v}y)_{s,t}^{(\alpha)}=x^{\alpha}(1\oplus_{u,v}(y/x))_{s,t}^{(\alpha)}$, then it is sufficient to prove our result for the series representation of $(1\oplus_{u,v}x)_{s,t}^{(\alpha)}$. 

\begin{theorem}
Set $s\neq$, $t\neq0$ and pick $u\neq0$, $\vert q\vert=\vert\varphi_{s,t}^\prime/\varphi_{s,t}\vert\neq1$ . For all $\alpha\in\C$ we have that
\begin{enumerate}
    \item  If $\vert uv\vert<\vert t\vert$ and $\vert q\vert\neq1$, then $(1\oplus_{u,v}x)_{s,t}^{(\alpha)}$ is a function entire.
    \item If $\vert uv\vert=\vert t\vert$, then $(1\oplus_{u,v}x)_{s,t}^{(\alpha)}$ is convergent in $\vert x\vert<\vert(\varphi_{s,t}/\varphi_{s,t}^\prime)(u/\varphi_{s,t}^{\prime})^{\alpha-1}\vert$ if $0<\vert q\vert<1$ or convergent in $\vert x\vert<\vert(\varphi_{s,t}^{\prime}/u)(u/\varphi_{s,t})^{\alpha}\vert$ if $\vert q\vert>1$.
    \item If $\vert uv\vert>\vert t\vert$ and $\vert q\vert\neq1$, then $(1\oplus_{u,v}x)_{s,t}^{(\alpha)}$ converge in $x=0$.
\end{enumerate}
\end{theorem}
\begin{proof}
By applying the criteria of D'Alembert to the series representation of $(1\oplus_{u,v}x)_{s,t}^{(\alpha)}$ we have
\begin{align*}
    \lim_{k\rightarrow\infty}\Bigg\vert\frac{\fibonomial{\alpha}{k+1}_{s,t}v^{\binom{k+1}{2}}u^{\binom{\alpha-k-1}{2}}x^{k+1}}{\fibonomial{\alpha}{k}_{s,t}v^{\binom{k}{2}}u^{\binom{\alpha-k}{2}}x^k}\Bigg\vert&=\lim_{k\rightarrow\infty}\Bigg\vert\frac{x\brk[c]{\alpha-k}_{s,t}v^{k}}{\brk[c]{k+1}_{s,t}u^{\alpha-k-1}}\Bigg\vert\\
    &=\Bigg\vert\frac{x}{u^{\alpha-1}}\Bigg\vert\lim_{k\rightarrow\infty}\vert uv\vert^k\Bigg\vert\frac{\brk[c]{\alpha-k}_{s,t}}{\brk[c]{k+1}_{s,t}}\Bigg\vert\\
    &=\Bigg\vert\frac{x}{u^{\alpha-1}}\Bigg\vert\lim_{k\rightarrow\infty}\vert uv\vert^k\Bigg\vert\frac{\varphi_{s,t}^{\alpha-k}-\varphi_{s,t}^{\prime(\alpha-k)}}{\varphi_{s,t}^{k+1}-\varphi_{s,t}^{\prime(k+1)}}\Bigg\vert\\
    &=\Bigg\vert\frac{x}{u^{\alpha-1}}\Bigg\vert\lim_{k\rightarrow\infty}\vert uv\vert^k\vert\varphi_{s,t}^{\alpha-2k-1}\vert\Bigg\vert\frac{1-q^{\alpha-k}}{1-q^{k+1}}\Bigg\vert\\
    &=\Bigg\vert\frac{x\varphi_{s,t}^{\alpha-1}}{u^{\alpha-1}}\Bigg\vert\lim_{k\rightarrow\infty}\Bigg\vert\frac{uv}{\varphi_{s,t}^2}\Bigg\vert^k\Bigg\vert\frac{1-q^{\alpha-k}}{1-q^{k+1}}\Bigg\vert.
\end{align*}
Set $A=uv/\varphi_{s,t}^2$ and define
\begin{align*}
    H_{k}=\vert A\vert^k\Bigg\vert\frac{1-q^{\alpha-k}}{1-q^{k+1}}\Bigg\vert&=\Bigg\vert\frac{A}{q}\Bigg\vert^k\Bigg\vert\frac{1}{q^{-k}-q}-\frac{q^{\alpha}}{1-q^{k+1}}\Bigg\vert.
\end{align*}
If $\vert A\vert<\vert q\vert$ and if $0<\vert q\vert<1$ or if $\vert q\vert>1$, then $H_{k}\mapsto0$ as $k\mapsto\infty$ and therefore $x\in\C$ for each $\alpha\in\C$. If $\vert A\vert=\vert q\vert$, then $H_{k}\mapsto\vert q\vert^\alpha$ if $\vert q\vert<1$ and $H_{k}\mapsto\vert q\vert^{-1}$ if $\vert q\vert>1$ as $k\mapsto\infty$. Finally, if $\vert A\vert>\vert q\vert$, then $H_{k}\mapsto\infty$ as $k\mapsto\infty$. The three statements follow from the above.
\end{proof}

\begin{theorem}
Set $s\neq$, $t\neq0$ and pick $u\neq0$. For all $\alpha\in\C$ we have that
\begin{enumerate}
    \item  If $\vert uv\vert<1$, then $(1\oplus_{u,v}x)^{(\alpha)}$ is a function entire.
    \item If $\vert uv\vert=1$, then $(1\oplus_{u,v}x)^{(\alpha)}$ is convergent in $\vert x\vert<\vert u^{\alpha-1}\vert$.
    \item If $\vert uv\vert>1$, then $(1\oplus_{u,v}x)^{(\alpha)}$ converge in $x=0$.
\end{enumerate}
\end{theorem}

The analytical importance of research into series representation with deformations lies in their convergence. For example, $f(z)=1/(1-z)$ has a convergent series representation when $\vert z\vert<1$, $z\in\C$. While $(1\ominus_{1,v}z)^{(-1)}$, the $(1,v)$-deformed analog of $f(z)$, has a convergent series representation in all $\C$ whenever $\vert v\vert<1$.

\begin{theorem}\label{theo_der_bino_a}
For all $\alpha\in\C$
\begin{enumerate}
    \item $\mathbf{D}_{s,t}(x\oplus_{u,v}a)_{s,t}^{(\alpha)}=\brk[c]{\alpha}_{s,t}(ux\oplus_{u,v}a)_{s,t}^{(\alpha-1)}$.
    \item $\mathbf{D}_{s,t}(a\oplus_{u,v}x)_{s,t}^{(\alpha)}=\brk[c]{\alpha}_{s,t}(a\oplus_{u,v}vx)_{s,t}^{(\alpha-1)}$.
    \item $\mathbf{D}_{s,t}(a\ominus_{u,v}x)_{s,t}^{(\alpha)}=-\brk[c]{\alpha}_{s,t}(a\oplus_{u,v}vx)_{s,t}^{(\alpha-1)}$.
\end{enumerate}
\end{theorem}
\begin{proof}
We prove assertion 1. The other assertions are proved similarly.
\begin{align*}
    \mathbf{D}_{s,t}(x\oplus_{u,v}a)_{s,t}^{(\alpha)}&=\mathbf{D}_{s,t}\left(\sum_{k=0}^{\infty}\fibonomial{\alpha}{k}_{s,t}u^{\binom{k}{2}}v^{\binom{\alpha-k}{2}}x^{k}a^{\alpha-k}\right)\\
    &=\sum_{k=1}^{\infty}\fibonomial{\alpha}{k}_{s,t}u^{\binom{k}{2}}v^{\binom{\alpha-k}{2}}\brk[c]{k}_{s,t}x^{k-1}a^{\alpha-k}\\
    &=\sum_{k=0}^{\infty}\fibonomial{\alpha}{k+1}_{s,t}u^{\binom{k}{2}}u^{k}v^{\binom{\alpha-k-1}{2}}\brk[c]{k+1}_{s,t}x^{k}a^{\alpha-k-1}\\
    &=\sum_{k=0}^{\infty}\frac{\brk[c]{\alpha}_{s,t}\brk[c]{\alpha-1}_{s,t}\cdots\brk[c]{\alpha-k}_{s,t}}{\brk[c]{k}_{s,t}!}u^{\binom{k}{2}}v^{\binom{\alpha-1-k}{2}}(u x)^{k}a^{\alpha-1-k}\\
    &=\brk[c]{\alpha}_{s,t}\sum_{k=0}^{\infty}\frac{\brk[c]{\alpha-1}_{s,t}\cdots\brk[c]{\alpha-1-k+1}_{s,t}}{\brk[c]{k}_{s,t}!}u^{\binom{k}{2}}v^{\binom{\alpha-1-k}{2}}(u x)^{k}a^{\alpha-1-k}\\
    &=\brk[c]{\alpha}_{s,t}(ux\oplus_{u,v}a)^{(\alpha-1)}.
\end{align*}    
\end{proof}

When set $u=\varphi_{s,t}$ and $v=\varphi_{s,t}^\prime$ in Theorem \ref{theo_der_bino_a}, then
\begin{equation*}
    \mathbf{D}_{s,t}(x\oplus_{\varphi,\varphi^\prime}\alpha)_{s,t}^{(\alpha)}=\brk[c]{\alpha}_{s,t}(\varphi_{s,t}x\oplus_{\varphi,\varphi^\prime}\alpha)_{s,t}^{(\alpha-1)}.
\end{equation*}

\begin{theorem}\label{theo_diff_k_binom}
For all $\alpha,u, v,a\in\C$
\begin{align*}
    \mathbf{D}^{k}_{s,t}(x\oplus_{u,v}a)_{s,t}^{(\alpha)}&=u^{\binom{k}{2}}\brk[c]{k}_{s,t}!\fibonomial{\alpha}{k}_{s,t}(u^{k}x\oplus_{u,v}a)_{s,t}^{(\alpha-k)},\\
    \mathbf{D}^{k}_{s,t}(a\oplus_{u,v}x)_{s,t}^{(\alpha)}&=u^{\binom{k}{2}}\brk[c]{k}_{s,t}!\fibonomial{\alpha}{k}_{s,t}(a\oplus_{u,v}v^4)_{s,t}^{(\alpha-k)},\\
    \mathbf{D}^{k}_{s,t}(a\ominus_{u,v}x)_{s,t}^{(\alpha)}&=(-1)^{k}v^{\binom{k}{2}}\brk[c]{k}_{s,t}!\fibonomial{\alpha}{k}_{s,t}(a\ominus_{u,v}v^{k}x)_{s,t}^{(\alpha-k)}.
\end{align*}
\end{theorem}
\begin{proof}
They are obtained by induction on $k$.
\end{proof}

\subsection{Proportional functional difference equation}

\begin{remark}
As
\begin{equation}
    (1\oplus_{u,v}x)_{s,t}^{(\alpha)}=u^{\binom{\alpha}{2}}(1\oplus_{1,uv}u^{-\alpha+1}x)_{s,t}^{(\alpha)},
\end{equation}
then it is sufficient to research the $(1,v)$-deformed $(s,t)$-binomial series $(1\oplus_{1,v}x)_{s,t}^{(\alpha)}$.
\end{remark}

\begin{theorem}
The function $f(x)=(1\oplus_{1,v}x)_{s,t}^{(\alpha)}$ is solution of the proportional functional difference equation
\begin{equation}\label{eqn_gpfde}
    (\mathbf{D}_{s,t}f)(\varphi_{s,t}x/v)+\varphi_{s,t}^{\prime(\alpha-1)}x(\mathbf{D}_{s,t}f)(x/\varphi_{s,t}^{\prime})=\brk[c]{\alpha}_{s,t}f(x),
\end{equation}
with initial value $f(0)=1$. If $v=\varphi_{s,t}\varphi_{s,t}^{\prime}$, then
\begin{equation}\label{eqn_ppfde}
    (\mathbf{D}_{s,t}f)(x/\varphi_{s,t}^{\prime})=\frac{\brk[c]{\alpha}_{s,t}f(x)}{1+\varphi_{s,t}^{\prime(\alpha-1)}x}.
\end{equation}
If $s=2$ and $t=-1$, then $f(x)=(1\oplus_{1,v}x)^{(\alpha)}$ is solution of the functional differential equation with proportional delay
\begin{equation}\label{eqn_pfde}
    f^\prime(x/v)+xf^\prime(x)=\alpha f(x).
\end{equation}
\end{theorem}
\begin{proof}
Use the Theorems \ref{theo_propi_uvFbinom} and \ref{theo_der_bino_a}.
\end{proof}
The solution of Eq.(\ref{eqn_ppfde}) is
\begin{equation*}
    f(x)=\varphi_{s,t}^{-\binom{n}{2}}\frac{(1\oplus_{\varphi,\varphi^\prime}\varphi_{s,t}^{\alpha-1}x)^{(\infty)}}{(\varphi_{s,t}^\alpha\oplus_{\varphi,\varphi^\prime}(-t)^\alpha\varphi_{s,t}^{-1}x)^{(\infty)}}.
\end{equation*}
The above theorem shows another justification for using deformed series representation. There is growing research on differential equations with proportional delay of the form $x^\prime(t)=f(t,x,x(qt))$, or pantograph equations, \cite{ebaid}, \cite{ise_1}, \cite{kato}. Eq.(\ref{eqn_pfde}) is part of the more general differential equation with proportional delay
\begin{equation*}
    f^{\prime}(x)=F(x,f,f(qx),f^{\prime}(qx)).
\end{equation*}
As Eq.(\ref{eqn_gpfde}) is the $(s,t)$-analog of Eq.(\ref{eqn_pfde}), then it opens the way to research proportional difference equations defined on generalized Fibonacci polynomials.

\section{The central $(s,t)$-binomial coefficients}

The first thing to do is to establish some properties of the generalized central binomial coefficient $\fibonomial{2n}{n}_{s,t}$ and in particular to find an expression for the coefficients $\fibonomial{1/2}{n}_{s,t}$ and $\fibonomial{-1/2}{n}_{s,t}$.

\begin{definition}\label{def_ln}
For all $n\in\N$ we define the functions $L_{n}(s,t)$ as
\begin{equation}
    L_{n}(s,t)=\frac{(-t)^{-\frac{n^2}{2}}4^n}{\brk[a]{n}_{s,t}!(\sqrt{\varphi_{s,t}}\oplus_{\varphi,\varphi^\prime}\sqrt{\varphi_{s,t}^\prime})_{s,t}^{(n)}}
\end{equation}
\end{definition}
The function $L_{n}(s,t)$ reduces to unity when $s=2$ and $t=-1$, that is,
\begin{align*}
    L_{n}(2,-1)&=\frac{4^n}{\brk[a]{n}_{2,-1}!2^{n}}=1
\end{align*}
for all $n\geq0$. For the case $s=1+q$, $t=-q$, we get
\begin{equation*}
    L_{n}(q)\equiv L_{n}(1+q,-q)=\frac{q^{-\frac{n^2}{2}}4^n}{(-q;q)_{n}(-\sqrt{q};q)_{n}}.
\end{equation*}
\begin{proposition}
Set $q=\varphi_{s,t}^{\prime}/\varphi_{s,t}$. For all $n\in\N$
    \begin{equation}
        L_{n+1}(s,t)=\frac{4(-t)^{-(n+1/2)}}{\brk[a]{n+1}_{s,t}(\varphi_{s,t}^{n+1/2}+\varphi_{s,t}^{\prime(n+1/2)})}L_{n}(s,t)
    \end{equation}
\end{proposition}
\begin{proof}
From Definition \ref{def_ln},
    \begin{align*}
        L_{n+1}(s,t)&=\frac{(-t)^{-\frac{(n+1)^2}{2}}4^{n+1}}{\brk[a]{n+1}_{s,t}!(\sqrt{\varphi_{s,t}}\oplus_{\varphi,\varphi^\prime}\sqrt{\varphi_{s,t}^\prime})_{s,t}^{(n+1)}}\\
        &=\frac{4(-t)^{-(n+1/2)}}{\brk[a]{n+1}_{s,t}(\varphi_{s,t}^{n+1/2}+\varphi_{s,t}^{\prime(n+1/2)})}\frac{(-t)^{-\frac{n^2}{2}}4^n}{\brk[a]{n}_{s,t}!(\sqrt{\varphi_{s,t}}\oplus_{\varphi,\varphi^\prime}\sqrt{\varphi_{s,t}^\prime})_{s,t}^{(n)}}\\
        &=\frac{4(-t)^{-(n+1/2)}}{\brk[a]{n+1}_{s,t}(\varphi_{s,t}^{n+1/2}+\varphi_{s,t}^{\prime(n+1/2)})}L_{n}(s,t)
    \end{align*}
\end{proof}

\begin{proposition}\label{prop_bino_(1/2)}
For all $n\geq1$
    \begin{equation}
        \fibonomial{1/2}{n}_{s,t}=\Big\{\frac{1}{2}\Big\}_{s,t}\fibonomial{2n}{n}_{s,t}\frac{(-1)^{n+1}}{4^{n-1}\brk[a]{n}_{s,t}\brk[c]{2n-1}_{s,t}}L_{n-1}(s,t).
    \end{equation}
\end{proposition}
\begin{proof}
From Definition \ref{def_fibo_real} and Eq.(\ref{eqn_neg_fibo}),
    \begin{align*}
        \fibonomial{1/2}{n}_{s,t}&=\frac{\brk[c]{1/2}_{s,t}\brk[c]{-1/2}_{s,t}\cdots\brk[c]{-(2n-3)/2}_{s,t}}{\brk[c]{n}_{s,t}!}\\
        &=\frac{\brk[c]{1/2}_{s,t}(-(-t)^{-1/2}\brk[c]{1/2}_{s,t})\cdots(-(-t)^{-(2n-3)/2}\brk[c]{(2n-3)/2}_{s,t})}{\brk[c]{n}_{s,t}!}\\
        &=\frac{(-1)^{n-1}(-t)^{-(\frac{1}{2}+\frac{3}{2}+\cdots+\frac{2n-3}{2})}\brk[c]{1/2}_{s,t}^{2}\brk[c]{3/2}_{s,t}\cdots\brk[c]{(2n-3)/2}_{s,t}}{\brk[c]{n}_{s,t}!}.
\end{align*}
As $\big\{\frac{2n+1}{2}\big\}_{s,t}=\frac{\brk[c]{2n+1}_{s,t}}{\brk[a]{(2n+1/2)}_{s,t}}$, then
\begin{align*}
    \Big\{\frac{1}{2}\Big\}_{s,t}\Big\{\frac{3}{2}\Big\}_{s,t}\cdots\Big\{\frac{2n-3}{2}\Big\}_{s,t}&=\frac{\brk[c]{1}_{s,t}\brk[c]{3}_{s,t}\cdots\brk[c]{2n-3}_{s,t}}{\brk[a]{1/2}_{s,t}\brk[a]{3/2}_{s,t}\cdots\brk[a]{(2n-3)/2}_{s,t}}\\
    &=\frac{\brk[c]{1}_{s,t}\brk[c]{3}_{s,t}\cdots\brk[c]{2n-3}_{s,t}}{(\varphi_{s,t}^{1/2}+\varphi_{s,t}^{\prime1/2})(\varphi_{s,t}^{3/2}+\varphi_{s,t}^{\prime3/2})\cdots(\varphi_{s,t}^{(2n-3)/2}+\varphi_{s,t}^{\prime(2n-3)/2})}\\
    &=\frac{\brk[c]{1}_{s,t}\brk[c]{3}_{s,t}\cdots\brk[c]{2n-3}_{s,t}}{(\sqrt{\varphi_{s,t}}\oplus_{\varphi,\varphi^\prime}\sqrt{\varphi_{s,t}^\prime})_{s,t}^{(n-1)}}
\end{align*}
and
\begin{align*}
    \fibonomial{1/2}{n}_{s,t}&=\Big\{\frac{1}{2}\Big\}_{s,t}\frac{(-1)^{n-1}(-t)^{-\frac{(n-1)^2}{2}}\brk[c]{1}_{s,t}\brk[c]{3}_{s,t}\cdots\brk[c]{2n-3}_{s,t}}{\brk[c]{n}_{s,t}!(\sqrt{\varphi_{s,t}}\oplus_{\varphi,\varphi^\prime}\sqrt{\varphi_{s,t}^\prime})_{s,t}^{(n-1)}}.        
\end{align*}
On the other hand,
\begin{align*}
    \brk[c]{2}_{s,t}\brk[c]{4}_{s,t}\cdots\brk[c]{2n}_{s,t}&=\brk[a]{1}_{s,t}\brk[c]{1}_{s,t}\brk[a]{2}_{s,t}\brk[c]{2}_{s,t}\cdots\brk[a]{n}_{s,t}\brk[c]{n}_{s,t}=\brk[a]{n}_{s,t}!\brk[c]{n}_{s,t}!.
\end{align*}
Therefore,
\begin{align*}
    \fibonomial{1/2}{n}_{s,t}&=\Big\{\frac{1}{2}\Big\}_{s,t}\frac{(-1)^{n-1}(-t)^{-\frac{(n-1)^2}{2}}\brk[c]{2n}_{s,t}!}{\brk[a]{n}_{s,t}!(\sqrt{\varphi_{s,t}}\oplus_{\varphi,\varphi^\prime}\sqrt{\varphi_{s,t}^\prime})_{s,t}^{(n-1)}\brk[c]{n}_{s,t}!\brk[c]{n}_{s,t}!\brk[c]{2n-1}_{s,t}}\\
    &=\Big\{\frac{1}{2}\Big\}_{s,t}\fibonomial{2n}{n}_{s,t}\frac{(-1)^{n-1}(-t)^{-\frac{(n-1)^2}{2}}}{\brk[a]{n}_{s,t}!(\sqrt{\varphi_{s,t}}\oplus_{\varphi,\varphi^\prime}\sqrt{\varphi_{s,t}^\prime})_{s,t}^{(n-1)}\brk[c]{2n-1}_{s,t}}\\
    &=\Big\{\frac{1}{2}\Big\}_{s,t}\fibonomial{2n}{n}_{s,t}\frac{(-1)^{n-1}}{4^{n-1}\brk[a]{n}_{s,t}\brk[c]{2n-1}_{s,t}}L_{n-1}(s,t).
\end{align*}
The proof is completed.
\end{proof}

\begin{proposition}\label{prop_bino_(-1/2)}
For all $n\in\N$,
    \begin{equation}
        \fibonomial{-1/2}{n}_{s,t}=\frac{(-1)^n}{4^{n}}\fibonomial{2n}{n}_{s,t}L_{n}(s,t).
    \end{equation}
\end{proposition}
\begin{proof}
From Definition \ref{def_fibo_real} and Eq.(\ref{eqn_neg_fibo}),
    \begin{align*}
        \fibonomial{-1/2}{n}_{s,t}&=\frac{\brk[c]{-1/2}_{s,t}\brk[c]{-3/2}_{s,t}\cdots\brk[c]{-(2n-1)/2}_{s,t}}{\brk[c]{n}_{s,t}!}\\
        &=\frac{(-(-t)^{-1/2}\brk[c]{1/2}_{s,t})(-(-t)^{-3/2}\brk[c]{3/2}_{s,t})\cdots(-(-t)^{-(2n-1)/2}\brk[c]{(2n-1)/2}_{s,t})}{\brk[c]{n}_{s,t}!}\\
        &=\frac{(-1)^{n}(-t)^{-(\frac{1}{2}+\frac{3}{2}+\cdots+\frac{2n-1}{2})}\brk[c]{1/2}_{s,t}\brk[c]{3/2}_{s,t}\cdots\brk[c]{(2n-1)/2}_{s,t}}{\brk[c]{n}_{s,t}!}\\
        &=\frac{(-1)^{n}(-t)^{-\frac{n^2}{2}}\brk[c]{1}_{s,t}\brk[c]{3}_{s,t}\cdots\brk[c]{2n-1}_{s,t}}{\brk[c]{n}_{s,t}!(\sqrt{\varphi_{s,t}}\oplus_{\varphi,\varphi^\prime}\sqrt{\varphi_{s,t}^\prime})_{s,t}^{(n)}}\\
        &=\frac{(-1)^{n}(-t)^{-\frac{n^2}{2}}\brk[c]{2n}_{s,t}!}{(\sqrt{\varphi_{s,t}}\oplus_{\varphi,\varphi^\prime}\sqrt{\varphi_{s,t}^\prime})_{s,t}^{(n)}\brk[a]{n}_{s,t}!\brk[c]{n}_{s,t}!\brk[c]{n}_{s,t}!}\\
        &=\frac{(-1)^n}{4^n}\fibonomial{2n}{n}_{s,t}L_{n}(s,t).
    \end{align*}
The expected result is reached.
\end{proof}

\section{Generating functions of central $(s,t)$-binomial coefficients}

\subsection{Deformed $(s,t)$-binomial series with rational powers}

\begin{definition}\label{def_roots}
For all non-zero real number $v$, $v\in\R$, such that $\vert v\vert\leq\vert t\vert$ and for all $\frac{n}{m}\in\Q$, $\frac{n}{m}>0$, we define the $(1,v)$-deformed $(s,t)$-binomial series with rational powers as
    \begin{equation}
        \sqrt[(m)]{(1\oplus_{1,v}x)^{(n)}_{s,t}}\equiv(1\oplus_{1,v}x)_{s,t}^{(n/m)}=\sum_{k=0}^{\infty}\fibonomial{n/m}{k}_{s,t}v^{\binom{k}{2}}x^{k}
    \end{equation}
and we define the "reciprocal" $(1,v)$-deformed $(s,t)$-binomial series as
    \begin{equation}
        \sqrt[(m)]{(1\oplus_{1,v}x)^{(-n)}_{s,t}}\equiv(1\oplus_{1,v}x)_{s,t}^{(-n/m)}=\sum_{k=0}^{\infty}\fibonomial{-n/m}{k}_{s,t}v^{\binom{k}{2}}x^{k}.
    \end{equation}
\end{definition}

\begin{proposition}
For all $\frac{n}{m}\in\Q$, 
    \begin{enumerate}
        \item $\sqrt[(n)]{(1\oplus_{1,v}x)^{(n)}_{s,t}}=1+x$.
        \item $\sqrt[(m)]{(1\oplus_{1,v}x)^{(nm)}_{s,t}}=(1\oplus_{1,v}x)^{(n)}_{s,t}$.
        \item $\sqrt[(m)]{(1\oplus_{1,v}x)_{s,t}^{(n+m)}}=\sqrt[(m)]{(1\oplus_{1,v}\varphi_{s,t}x)_{s,t}^{(n)}}+x\sqrt[(m)]{(\varphi_{s,t}^\prime\oplus_{1,v}vx)_{s,t}^{(n)}}$.
    \end{enumerate}
\end{proposition}
From Theorem \ref{theo_diff_k_binom} we have the following results.
\begin{theorem}
For all $k\in\N$,
\begin{equation*}
    \mathbf{D}_{s,t}^{k}\sqrt[(m)]{(1\oplus_{1,v}x)_{s,t}^{(n)}}=
    v^{\binom{k}{2}}\brk[c]{k}_{s,t}!\fibonomial{n/m}{k}_{s,t}\sqrt[(m)]{(1\oplus_{1,v}v^{k}x)_{s,t}^{(n-mk)}}.
\end{equation*}
\end{theorem}

\begin{theorem}
For $\vert v\vert\leq\vert t\vert$, the $v$-deformed series representation of $\sqrt[(2)]{(1\oplus_{1,v}4x)_{s,t}}$ is
    \begin{equation}\label{eqn_sqrt}
        \sqrt[(2)]{(1\oplus_{1,v}4x)_{s,t}}=1+4\Big\{\frac{1}{2}\Big\}_{s,t}\sum_{n=1}^{\infty}\fibonomial{2n}{n}_{s,t}\frac{(-1)^{n-1}L_{n-1}(s,t)}{\brk[a]{n}_{s,t}\brk[c]{2n-1}_{s,t}}v^{\binom{n}{2}}x^{n}.
    \end{equation}
\end{theorem}
\begin{proof}
By Definition \ref{def_roots} and Proposition \ref{prop_bino_(1/2)}
    \begin{align*}
        \sqrt[(2)]{(1\oplus_{1,v}x)_{s,t}}&=1+\sum_{n=1}^{\infty}\fibonomial{1/2}{n}_{s,t}v^{\binom{n}{2}}x^{n}\\
        &=1+\Big\{\frac{1}{2}\Big\}_{s,t}\sum_{n=1}^{\infty}\fibonomial{2n}{n}_{s,t}\frac{(-1)^{n-1}(-t)^{-\frac{(n-1)^2}{2}}v^{\binom{n}{2}}}{\brk[a]{n}_{s,t}!(\sqrt{\varphi_{s,t}}\oplus_{\varphi,\varphi^\prime}\sqrt{\varphi_{s,t}^{\prime}})_{s,t}^{(n-1)}\brk[c]{2n-1}_{s,t}}x^{n}.
    \end{align*}
Finish by changing $x$ with $4x$.
\end{proof}
If we put $x=1/4$ in Eq.(\ref{eqn_sqrt}), we get the series for the $(1,v)$-deformed $(s,t)$-analog of $\sqrt{2}$
\begin{align*}
    \sqrt[(2)]{(1\oplus_{1,v}1)_{s,t}}&=1+4\Big\{\frac{1}{2}\Big\}_{s,t}\sum_{n=1}^{\infty}\fibonomial{2n}{n}_{s,t}\frac{(-1)^{n-1}}{\brk[a]{n}_{s,t}\brk[c]{2n-1}_{s,t}}L_{n-1}(s,t)v^{\binom{n}{2}}.
\end{align*}
The $(1,q)$-deformed $q$-analog is
\begin{align*}
    \sqrt[(2)]{(1\oplus_{1,q}1)_{q}}&=1+4\bigg[\frac{1}{2}\bigg]_{q}\sum_{n=1}^{\infty}(-1)^{n-1}\frac{(q^2;q^2)_{n}(q;q^2)_{n}}{(q;q)_{n}^2(-q;q)_{n}(-\sqrt{q};q)_{n}}\frac{q^{\frac{n-1}{2}}(1-q)}{1-q^{2n-1}}
\end{align*}
and the $(s,t)$-analog is
\begin{align*}
    \sqrt[(2)]{(1\oplus_{1,-t}1)_{s,t}}&=\varphi_{s,t}^{1/8}\frac{(1\oplus_{\varphi,\varphi^\prime}\varphi_{s,t}^{-1/2})^{(\infty)}_{s,t}}{(\sqrt{\varphi_{s,t}}\oplus_{\varphi,\varphi^\prime}\sqrt{\varphi_{s,t}^\prime}\varphi_{s,t}^{-1/2})^{(\infty)}_{s,t}}.
\end{align*}

\begin{theorem}
For $\vert v\vert\leq\vert t\vert$, the $v$-deformed series representation of $\sqrt[(2)]{(1\ominus_{1,v}4x)_{s,t}^{(-1)}}$ is
    \begin{equation}\label{eqn_ogf_cbc}
        \sqrt[(2)]{(1\ominus_{1,v}4x)_{s,t}^{(-1)}}=\sum_{n=0}^{\infty}\fibonomial{2n}{n}_{s,t}L_{n}(s,t)v^{\binom{n}{2}}x^n.
    \end{equation}
\end{theorem}
\begin{proof}
By Definition \ref{def_roots} and Proposition \ref{prop_bino_(-1/2)}
    \begin{align*}
        \sqrt[(2)]{(1\ominus_{1,v}4x)_{s,t}^{(-1)}}&=\sum_{n=0}^{\infty}(-1)^{n}\fibonomial{-1/2}{n}_{s,t}v^{\binom{n}{2}}x^{n}\\
        &=\sum_{n=0}^{\infty}\fibonomial{2n}{n}_{s,t}L_{n}(s,t)v^{\binom{n}{2}}x^{n}.
    \end{align*}
The proof is reached.
\end{proof}
If we put $x=1/8$ in Eq.(\ref{eqn_ogf_cbc}), we get another series for the $v$-deformed $(s,t)$-analog of $\sqrt{2}$
\begin{equation*}
    \sqrt[(2)]{(1\ominus_{1,v}(1/2))_{s,t}^{(-1)}}=\sum_{n=0}^{\infty}\frac{1}{8^n}\fibonomial{2n}{n}_{s,t}L_{n}(s,t)v^{\binom{n}{2}}.
\end{equation*}
If we put $x=-1/8$, we get another series for the $v$-deformed $(s,t)$-analog of $\sqrt{2/3}$
\begin{equation*}
    \sqrt[(2)]{(1\oplus_{1,v}(1/2))_{s,t}^{(-1)}}=\sum_{n=0}^{\infty}\frac{(-1)^n}{8^n}\fibonomial{2n}{n}_{s,t}L_{n}(s,t)v^{\binom{n}{2}}.
\end{equation*}
Averaging the two above series we get the $v$-deformed $(s,t)$-analog of $\frac{3\sqrt{2}+\sqrt{6}}{6}$
\begin{equation*}
    \sqrt[(2)]{(1\oplus_{1,v}(1/2))_{s,t}^{(-1)}}+\sqrt[(2)]{(1\ominus_{1,v}(1/2))_{s,t}^{(-1)}}=\sum_{n=0}^{\infty}\frac{1}{64^n}\fibonomial{4n}{2n}_{s,t}L_{2n}(s,t)v^{\binom{2n}{2}}.
\end{equation*}

\subsection{Generating function of generalized Catalan numbers}

\begin{theorem}
Set $\vert v\vert\leq\vert t\vert$. The $v$-deformed generating function of the Catalan polynomials $C_{\brk[c]{n}}$ is
    \begin{equation}\label{eqn_dgf_cat}
        \frac{1-\sqrt[(2)]{(1\ominus_{1,v}4x)_{s,t}}}{2x}=2\Big\{\frac{1}{2}\Big\}_{s,t}\sum_{n=0}^{\infty}C_{\{n\}}L_{n}(s,t)v^{\binom{n+1}{2}}x^n.
    \end{equation}
\end{theorem}
\begin{proof}
From Eq.(\ref{eqn_sqrt}),
\begin{align*}
    \frac{1-\sqrt[(2)]{(1\ominus_{1,v}4x)_{s,t}}}{2x}&=\frac{2}{x}\Big\{\frac{1}{2}\Big\}_{s,t}\sum_{n=1}^{\infty}\fibonomial{2n}{n}_{s,t}\frac{(-1)^{n}L_{n-1}(s,t)}{\brk[a]{n}_{s,t}\brk[c]{2n-1}_{s,t}}v^{\binom{n}{2}}(-1)^nx^{n}\\
    &=2\Big\{\frac{1}{2}\Big\}_{s,t}\sum_{n=0}^{\infty}\fibonomial{2n+2}{n+1}_{s,t}\frac{L_{n}(s,t)}{\brk[a]{n+1}_{s,t}\brk[c]{2n+1}_{s,t}}v^{\binom{n+1}{2}}x^{n}\\
    &=2\Big\{\frac{1}{2}\Big\}_{s,t}\sum_{n=0}^{\infty}\fibonomial{2n}{n}_{s,t}\frac{L_{n}(s,t)}{\brk[c]{n+1}_{s,t}}v^{\binom{n+1}{2}}x^{n}.
\end{align*}
The proof is completed.
\end{proof}
Then
\begin{equation*}
    \sum_{n=0}^{\infty}\frac{C_{\brk[c]{n}}}{4^n}L_{n}(s,t)v^{\binom{n+1}{2}}=\frac{1-\sqrt[(2)]{(1\ominus_{1,v}1)_{s,t}}}{\brk[c]{1/2}_{s,t}}
\end{equation*}
is the $v$-deformed $(s,t)$-analog of $\sum_{n=0}^{\infty}\frac{C_{n}}{4^n}=2$.

\begin{theorem}
Set $\vert v\vert\leq\vert t\vert$. The $v$-deformed generating function of the polynomials $\brk[c]{n}_{s,t}C_{\brk[c]{n}}$ is
\begin{multline*}
    \frac{4\brk[c]{1/2}_{s,t}\varphi_{s,t}x-(1-\sqrt[(2)]{(1\ominus_{1,v}4\varphi_{s,t}x)_{s,t}})\sqrt[(2)]{(1\ominus_{1,v}4vx)_{s,t}}}{2\varphi_{s,t}\varphi_{s,t}^\prime\sqrt[(2)]{(1\ominus_{s,t}4vx)_{s,t}}}\\
    =2\Big\{\frac{1}{2}\Big\}_{s,t}\sum_{n=0}^{\infty}C_{\brk[c]{n}}\brk[c]{n}_{s,t}L_{n}(s,t)v^{\binom{n+1}{2}}x^n.
\end{multline*}    
\end{theorem}
\begin{proof}
Apply the operator $x\mathbf{D}_{s,t}$ to both sides of Eq.(\ref{eqn_dgf_cat}).  
\end{proof}

\begin{theorem}
Set $\vert v\vert\leq\vert t\vert$. The $v$-deformed generating function of the polynomials $C_{\brk[c]{n}_{s,t}}\brk[c]{2n+1}_{s,t}\brk[a]{n+1}_{s,t}$ is
    \begin{equation}
        \frac{(1\ominus_{1,v}4x)_{s,t}^{(-1/2)}-1}{2x}=\frac{1}{2}\sum_{n=0}^{\infty}C_{\brk[c]{n}_{s,t}}\brk[c]{2n+1}_{s,t}\brk[a]{n+1}_{s,t}L_{n+1}(s,t)v^{\binom{n+1}{2}}x^{n}.
    \end{equation}
\end{theorem}
\begin{proof}
From Eq.(\ref{eqn_ogf_cbc})
\begin{align*}
    \frac{(1\ominus_{1,v}4x)_{s,t}^{(-1/2)}-1}{2x}&=\frac{1}{2x}\sum_{n=1}^{\infty}\fibonomial{2n}{n}_{s,t}L_{n}(s,t)v^{\binom{n}{2}}x^{n-1}\\
    &=\frac{1}{2}\sum_{n=0}^{\infty}\fibonomial{2n+2}{n+1}_{s,t}L_{n+1}(s,t)v^{\binom{n+1}{2}}x^{n}\\
    &=\frac{1}{2}\sum_{n=0}^{\infty}\fibonomial{2n}{n}_{s,t}\frac{\brk[c]{2n+1}_{s,t}\brk[a]{n+1}_{s,t}}{\brk[c]{n+1}_{s,t}}L_{n+1}(s,t)v^{\binom{n+1}{2}}x^{n}.
\end{align*}
The proof is completed.
\end{proof}

\end{document}